\documentclass{amsart}[12pt,article]

\usepackage{amsmath,amssymb,amscd,amsthm,indentfirst}
\usepackage{amsfonts,eufrak}
\usepackage[mathscr]{eucal}
\usepackage[all]{xy}
\usepackage{enumerate}

\newtheorem*{teoA}{Theorem A}
\newtheorem*{teoB}{Theorem B}
\newtheorem*{teoC}{Theorem C}

\newtheorem{proposition}{Proposition}
\newtheorem{theorem}[proposition]{Theorem}
\newtheorem{corollary}[proposition]{Corollary}
\newtheorem{lemma}[proposition]{Lemma}

\newtheorem{remark}[proposition]{Remark}

\theoremstyle{definition}

\def\Mor{\mathrm{Mor}}

\def\Aut{\mathrm{Aut}}
\def\Obj{\mathrm{Obj}}

\def\M{\ensuremath{\mathcal{M}}}

\def\FF{\ensuremath{\mathbb{F}}}

\def\ZZ{\ensuremath{\mathbb{Z}}}

\def\End{\operatorname{End}}

\def\Aut{\mathrm{Aut}}

\def\ker{\operatorname{Ker}}

\def\Hom{\mathrm{Hom}}

\def\id{\textrm{id}}

\def\im{\operatorname{Im}}

\date{\today}
\title{Transporting cohomology in Lazard correspondence}
\author{Oihana Garaialde Oca\~{n}a}
 \email{ogaraialde@gmail.com}
 \address{Departamento de Matem\'aticas\\
Facultad de Ciencia y Tecnolog�a\\
UPV/EHU University of the Basque Country
}

\author{Jon Gonz\'alez-S\'anchez}
 \email{jon.gonzalez@ehu.es}
 \address{Jon Gonz\'alez-S\'anchez\\
Departamento de Matem\'aticas\\
Facultad de Ciencia y Tecnolog�a\\
UPV/EHU University of the Basque Country
}

\begin{document}

\begin{abstract} 
 Lazard correspondence provides an isomorphism of categories between finitely generated nilpotent pro-$p$ groups of nilpotency class smaller than $p$ and finitely generated nilpotent $\ZZ_p$-Lie algebras of nilpotency class smaller than $p$. Denote by $H_{Gr}^i$ and $H_{Lie}^i$ the group cohomology functors and the Lie cohomology functors respectively. The aim of this paper is to show that for $i=0$, $1$ and $2$, and for a given category of modules the cohomology functors $H_{Gr}^i\circ \textbf{exp}$ and $H^i_{Lie}$ are naturally equivalent. A similar result is proven for $i=3$ and the relative cohomology groups.
\end{abstract} 

\maketitle

\section{Introduction}
Throughout the text let $p$ denote a fixed prime and $\ZZ_p$ the $p$-adic integers. In his work of 1954 M. Lazard established an isomorphism of categories between finitely generated nilpotent pro-$p$ groups of nilpotency class smaller than $p$ and finitely generated nilpotent $\ZZ_p$-Lie algebras of nilpotency class smaller than $p$ (see \cite{La}). This well known result of Lazard has had strong influence in the study of finite $p$-groups. For example, classifying  ``small" finite $p$-groups is equivalent to classifying ``small" nilpotent $\ZZ_p$-Lie algebras which turned out to be fundamental in the classification of ``small" finite $p$-groups (see \cite{OV}). This isomorphism of categories is given by the exponential and logarithm functors. In the context of complex representations, the orbit method of A. Kirillov can be applied to nilpotent pro-$p$ groups of nilpotency class smaller than $p$. This method establishes a bijection between the complex irreducible representations of the group and the orbits of the coadjoint action of the group in the dual space of its Lie algebra (see \cite{Ka} and \cite{GS}). 

If $G$ is a finite $p$-group and $K$ is a field of characteristic $p$, there is only one irreducible representation over $K$, the trivial one. Therefore irreducible representations over $K$ do not give information about the group. However, the homology and cohomology groups provide more information about the group and about the modules. For example, if $G$ is a finite $p$-group and $\FF_p$ denotes the trivial module, $H^1(G,\FF_p)$ gives the minimal number of generators of $G$ and $H^2(G,\FF_p)$ the minimal number of relations.

In a recent paper of B. Eick, M. Horn and S. Zandi it has been proved that if $G$ is finite $p$-group of nilpotency class smaller than $p-1$ and $L$ is its Lie algebra then, the Schur multiplier of $G$ and $L$ are isomorphic (see \cite{EHZ}). Note that the Schur multipliers of $G$ and $L$ are given by the cohomology groups $H^2(G,C_{p^\infty})$ and $H^2(L,C_{p^\infty})$ respectively. 

The aim of this text is to study cohomology groups of small dimensions in the context of the Lazard correspondence. In order to do it we will introduce certain  categories of modules over groups $\mathbf{Tpl}_{Gr}^{c,d}$ (see Section \ref{s:groups}) and certain categories of modules over Lie algebras $\mathbf{Tpl}_{Lie}^{c,d}$ (see Section \ref{s:Lie}). In these categories $c$ denotes the nilpotency class of the group or of the Lie algebra while $d$ denotes the length action of the group or of the Lie algebra over the corresponding module. In our first result we give an isomorphism of categories using truncated exponential and logarithm functors.

\begin{teoA}
Let $c$ and $d$ be smaller than $p$. Then, 
$$\mathbf{Exp}: \mathbf{Tpl}_{Lie}^{c,d} \rightarrow  \mathbf{Tpl}_{Gr}^{c,d} \quad  \text{and} \quad \mathbf{Log}: \mathbf{Tpl}_{Gr}^{c,d} \rightarrow \mathbf{Tpl}_{Lie}^{c,d}$$ 
are isomorphisms inverse to each other. 
\end{teoA}

We use this result to describe cohomology groups  of the pro-$p$ group as cohomology groups of the Lie algebra.

\begin{teoB}\label{theoremB} Denote by  $H^i_{Lie}:\mathbf{Tpl}_{Lie}^{c,d} \longrightarrow \textbf{Ab}$ and $H^i_{Lie}:\mathbf{Tpl}_{Gr}^{c,d} \longrightarrow \textbf{Ab}$ the natural cohomology functors. Then
\begin{enumerate}
\item If  $c,d <p$, then $H_{Lie}^0$ and $H_{Gr}^0 \circ \mathbf{Exp}$ are naturally equivalent. 
\item If $c<p$ and $d<p-1$, then $H_{Lie}^1$ and  $H_{Gr}^1 \circ \mathbf{Exp}$ are naturally equivalent.
\item If $c+d<p$, then $H_{Lie}^2$ and $H_{Gr}^2 \circ \mathbf{Exp}$ are naturally equivalent.
\end{enumerate}
\end{teoB}

In degree three cohomology we give a similar result but for relative cohomology (see Proposition \ref{cross3}).

Theorem B.(3) provides a natural proof of the result of  B. Eick, M. Horn and S. Zandi on the Schur multipliers of the groups and the Lie algebras.

\begin{teoC}Let $\textbf{Gr}_{p-2}$ be the category of a finite $p$-group of nilpotency class smaller than $p-1$ and $\textbf{Lie}_{p-2}$ the category of finite and nilpotent $\ZZ_p$-Lie algebras of nilpotency class smaller than $p-1$. Denote by 
\begin{eqnarray*}
  \M_{Gr} & : & \textbf{Gr}_{p-2}\longrightarrow \textbf{Ab} \ \ \ \text{and} \\
  \M_{Lie} & :  & \textbf{Lie}_{p-2}\longrightarrow \textbf{Ab} 
\end{eqnarray*}
the group and Lie algebra Schur Multiplier functors respectively. Then 
$\M_{Gr}\circ \textbf{exp}$ and $\M_{Lie}$ are naturally equivalent. In particular, for $L\in \textbf{Lie}_{p-2}$ one has $\M_{Gr}(\textbf{exp}(L))\cong \M_{Lie} (L)$.
\end{teoC}

The paper is divided as follows. In the second section we state the Lazard corresponde for pro-$p$ groups and Lie rings. In the third and fourth sections we introduce the categories of modules $\mathbf{Tpl}_{Gr}^{c,d}$ and $\mathbf{Tpl}_{Lie}^{c,d}$ and then we recall the definition of the cohomology groups of small dimension. In the fifth section we define the $\mathbf{Exp}$ and $\mathbf{Log}$ functors between the categories of modules and we prove Theorem A. In the next section we prove Theorem B and finally, in the last section, we prove Theorem C.

\section{Lazard correspondence for finite nilpotent pro-$p$ groups and Lie rings}

Let $X$ be a set and $A(X)$ denote the $\ZZ_p$-algebra of non-commuting polynomials over $X$. In fact, $A(X)$ is the free associative $\ZZ_p$-algebra over $X$. Consider $A_{(p)}(X)$ the ideal generated by the monomials of degree $p$. Then $A(X)_p=A(X)/A_{(p)}(X)$ is the free associative $\ZZ_p$-algebra of nilpotency class $p-1$ over $X$. The algebra $A(X)_p$ with the Lie bracket given by $[a,b]=ab-ba$ with $a,b\in A(X)_p$ is a $\ZZ_p$-Lie algebra and the $\ZZ_p$-Lie subalgebra $L(X)_p$ generated by $X$ is the free nilpotent $\ZZ_p$-Lie algebra of nilpotency class $p-1$. In $A(X)_p$ we can define the exponential and logarithm functions
\begin{eqnarray*}
exp (a) & = & \sum_{k=0}^{p-1} \frac{1}{k!} a^k\ \ \ \text{and} \\
log(1+a) & = & \sum_{k=1}^{p-1} (-1)^{k+1}\frac{a^k}{k},
\end{eqnarray*}
and the Baker-Campbell-Hausdorff formula
$$H(a,b) = log(exp(a) \cdot exp(b)).$$
By Lazard \cite{La} we know that $L(X)_p$ together with the multiplication coming from the Baker-Campbell-Hausdorff formula is the free pro-$p$ group of nilpotency class $p-1$ over $X$ which will be denoted by  $F(X)_p$. Furthermore, one can invert the Baker-Campbell-Hausdorff formula to get the Lie algebra structure from the group structure. Denote by
\begin{eqnarray*}
h_1(a, b) & = &  exp (log (a) +log (b))\ \ \   \text{and} \\
h_2(a,b) & = &  exp ([log (a),log(b)]).
\end{eqnarray*}
Then $L(X)_p$ is isomorphic to $(F(X)_p,h_1,h_2)$ as $\ZZ_p$-Lie algebras. Moreover this isomorphism can be extended to finitely generated nilpotent $\ZZ_p$-Lie algebras of nilpotency class smaller than $p$ and finitely generated nilpotent pro-$p$ groups of nilpotency class smaller than $p$ as it is stated in the next theorem.

\begin{theorem}[\cite{La} Lazard]
Let $Gr_p$ denote the category of finitely generated nilpotent pro-$p$ groups of nilpotency class smaller than $p$ and $Lie_p$ the category of finitely generated nilpotent $\ZZ_p$-Lie algebras of nilpotency class less than $p$. Then there exist isomorphisms of categories one inverse of the other 
\begin{eqnarray*}
\textbf{exp} : Lie_p\longrightarrow Gr_p \\
\textbf{log} :Gr_p\longrightarrow Lie_p,
\end{eqnarray*}
such that for $G\in Gr_p$ and $L\in Lie_p$ the following statements hold:
\begin{itemize}
\item[(a)] $\textbf{exp} (L)=(L,H)$,
\item[(b)] $\textbf{log} (G)=(G,h_1,h_2)$,
\item [(c)] $K$ is a subgroup of $G$ if and only if $log(K)$ is a sub-Lie algebra of $log(G)$,
\item[(d)] $K$ is a normal subgroup of $G$ if and only if $log(K)$ is an ideal in $log(G)$,
\item [(e)] Nilpotency class of $G$ = nilpotency class of $log(G)$,
\item [(f)] $\End(G)=\End(log(G))$. In particular, $\Aut(G) = \Aut(log(G))$.
\end{itemize}
\end{theorem}

\section{Group, modules and cohomology of low dimension}

\label{s:groups}
For a finite pro-$p$ group $G$, we recall that a $G$-module $M$ is a $\mathbb{Z}_p$-module  together with a homomorphism $\phi : G\to \Aut (M)$. We define the category $\mathbf{Tpl}_{Gr}$ that takes as objects the triples $(G,M,\phi )$ where $G$ is a finite pro-$p$ group, $M$ is a $\mathbb{Z}_p$-module and $\phi\in  \Hom (G,\Aut (M))$, that is, $M$ is a $\mathbb{Z}_p [G]$-module. Given two objects $(G_1,M_1,\phi_1 )$ and $(G_2,M_2,\phi_2 )$ in $\mathbf{Tpl}_{Gr}$, a morphism from $(G_1,M_1,\phi_1 )$ to $(G_2,M_2,\phi_2 )$ is defined by a pair $(\alpha ,\beta )$ where $\alpha \in \Hom (G_2,G_1)$, $\beta \in  \Hom (M_1,M_2)$ and for all $m_1 \in M_1$ and $g_2 \in G_2$ the following holds:

$$\beta((\phi_1 \circ \alpha(g_2))(m_1)) = (\phi_2(g_2))(\beta(m_1)).$$

\subsection{Defining $H^0$ in the category $\mathbf{Tpl}_{Gr}$} For a triple $(G,M,\phi )\in \mathbf{Tpl}_{Gr}$ the degree zero cohomology group is defined by
$$H^0(G,M,\phi)=H^0(G,M)=\{m \in M \; | \;  m \cdot \phi (g)= m, \; \forall g \in G\}.$$
Furthermore, if one has $(\alpha,\beta )$ a morphism between two triples $(G_1,M_1,\phi_1 )$ and $(G_2,M_2,\phi_2 )$, then 
$$\beta : H^0(G_1,M_1)\subseteq M_1\longrightarrow H^0(G_2,M_2)\subseteq M_2$$ 
is a homomorphism of abelian groups. In fact, 
$H^0(\cdot )$ defines a covariant functor between the categories $\mathbf{Tpl}_{Gr}$ 
and the category of abelian groups $\mathbf{Ab}$.

\subsection{Defining $H^1$ in the category $\mathbf{Tpl}_{Gr}$} For a triple $(G,M,\phi) \in \mathbf{Tpl}_{Gr}$, each equivalence class in $H^1(G,M,\phi)$ is in correspondence with the equivalent exact sequence of $G$-modules $0 \rightarrow M \rightarrow \tilde{M} \rightarrow \mathbb{Z}_p \rightarrow 0$. That is,

$$H^1(G,M,\phi)=\frac{\{\text{Equivalent extension of}\quad \mathbb{Z}_p[G]\text{-modules}\}}{\{\text{Equivalent split extensions of}\quad \mathbb{Z}_p[G]\text{-modules}\}}.$$

In order to give the additive structure over $H^1(G,M)$ one defines the sum of  any such two extensions $0 \rightarrow M \rightarrow \tilde{M}_1 \rightarrow \mathbb{Z}_p \rightarrow 0$ and $0 \rightarrow M \rightarrow \tilde{M}_2 \rightarrow \mathbb{Z}_p \rightarrow 0$ as follows:

\begin{equation}
 \label{eq:H1Gr-1}
     \xymatrix{ & & &0 \ar[d] &0 \ar[d] & \\
& & &M \ar[d]^{\tilde{f_1}} \ar[r]^{id} &M \ar[d]^{f_1} &\\
&0 \ar[r] &M \ar[d]^{id}\ar[r]^{\tilde{f_2}} &\tilde{M} \ar[d]^{i_2}\ar[r]^{i_1} &\tilde{M}_1 \ar[d]^{g_1} \ar[r] &0 \\
&0 \ar[r] &M \ar[d] \ar[r]^{f_2} &\tilde{M}_2 \ar[d] \ar[r]^{g_2} \ar[r] &\mathbb{Z}_p \ar[d] \ar[r] &0 \\
& &0 &0 &0 & }
\end{equation}
where $\tilde{M}\subseteq \tilde{M}_1 \oplus \tilde{M}_2$ is the pull-back of the arrows $\tilde{M}_1 \rightarrow \ZZ_p$ and $\tilde{M}_2 \rightarrow \ZZ_p$. Notice that $\tilde{f_1}$ and $\tilde{f_2}$ are defined by $\tilde{f_1}(m)=(f_1(m),0)$ and $\tilde{f_2}(m)=(0,f_2(m))$. Then, the following exact sequence is the sum of the above extensions:

$$0 \rightarrow \frac{M\oplus M}{\Delta^{-}(M)} \rightarrow \tilde{M} \rightarrow \mathbb{Z}_p \rightarrow 0$$
where $\Delta^{-}(M)$ denotes the anti-diagonal map of $M$ and $\frac{M \oplus M}{\Delta^{-}(M)}$ is isomorphic to $M$.

Furthermore, given $(G_1,M_1,\phi_1)$ and $(G_2,M_2,\phi_2)$ two objects in $\mathbf{Tpl}_{Gr}$ and a morphism $(\alpha ,\beta) \in \Mor((G_1,M_1,\phi_1), (G_2,M_2,\phi_2))$, there is an induced homomorphism in cohomology. 

Indeed, let $0 \rightarrow M_1 \rightarrow \tilde{M_1} \rightarrow \mathbb{Z}_p \rightarrow 0$ be an element in $H^1(G_1,M_1,\phi_1)$. On the one hand, it is easy to see that this short exact sequence of $G_1$-modules can be seen as a short exact sequence of $G_2$-modules by considering the action of $G_2$ over such modules given by the homomorphism $\phi_1\circ \alpha$. On the other hand, given $\beta: M_1 \rightarrow M_2$ consider the following diagram constructed by taking the push-out of the arrows $M_1 \rightarrow \tilde{M_1}$ and $M_1 \rightarrow M_2$:

\begin{equation}
\label{eq:H1Gr-2}
    \xymatrix{ 0 \ar[r] &M_1 \ar[d]^{\beta} \ar[r] &\tilde{M_1} \ar[d]^{\tilde{\beta}} \ar[r]^{j_1} &\mathbb{Z}_p\ar[d]^{id} \ar[r] &0 \\
                   0 \ar[r] &M_2 \ar[r] &\tilde{M_2} \ar[r]^{\theta} &\mathbb{Z}_p \ar[r] &0. }
\end{equation}
The homomorphism $\theta: M_2 \rightarrow \mathbb{Z}_p$ is given by $\theta(a,b)=j_1(a)$. In this way, the second row is an element in $H^1(G_2,M_2,\phi_2)$.

In fact one has that $H^1(\cdot )$ is a covariant functor between $\mathbf{Tpl}_{Gr}$ 
and $\mathbf{Ab}$.

\subsection{Defining $H^2$ in the category $\mathbf{Tpl}_{Gr}$} Each equivalence class of $H^2(G,M)$ classifies the equivalent extensions (see \cite[Chapt IV]{Br}) of $G$ by $M$ which consider the extensions of the form $1 \rightarrow M \rightarrow \tilde{G} \rightarrow G \rightarrow 1$, that is,

$$H^2(G,M)= \frac{\{\text{Equivalent extensions of G by M}\}}{\{\text{Equivalent split extensions of G by M}\}}.$$

By extensions of $G$ by $M$, we mean the above extensions that give rise to the given action of $G$ on $M$. The additive structure over $H^2(G,M)$ is given by the Baer sum \cite{Baer}. Namely, for two extensions of $G$ by $M$,

\begin{displaymath}
         \xymatrix{ 1 \ar[r] & M \ar[r]^{f_1} &\tilde{G_1} \ar[r]^{g_1} &G \ar[r] &1 }
\end{displaymath}
and 
\begin{displaymath}
         \xymatrix{ 1 \ar[r] & M \ar[r]^{f_2} &\tilde{G_2} \ar[r]^{g_2} &G \ar[r] &1 }
\end{displaymath}
the Baer sum is defined as follows: consider the next diagram

\begin{equation} \label{eq:H2Gr-1}
     \xymatrix{ & &1 \ar[d] &1 \ar[d] & \\ 
                  &  &M \ar[d] \ar[r]^{id} &M \ar[d]^{f_1} &  \\
                 1 \ar[r] &M \ar[d]^{id} \ar[r] &\tilde{H} \ar[d]^{i_2} \ar[r]^{i_1} &\tilde{G_1} \ar[d]^{g_1} \ar[r] &1 \\
                  1 \ar[r]&M \ar[r]^{f_2} &\tilde{G_2} \ar[d] \ar[r]^{g_2} &G \ar[d] \ar[r] &1 \\
                  & &1 &1 & }
\end{equation}
where $\tilde{H} \subset \tilde{G_1} \times \tilde{G_2}$ is the pull-back of the arrows $\tilde{G}_1\rightarrow G$ and $\tilde{G}_2\rightarrow G$. Now take 
$$\tilde{G}= \frac{\tilde{H}}{\{ (f_1(m),0)-(0,f_2(m)) \; | \; m \in M \}}.$$
Then, the Baer sum is the following extension:

\begin{equation*}
     \xymatrix{ 1 \ar[r] &M \ar[r]^{\tilde{f}} &\tilde{G} \ar[r]^{\tilde{g}} &G \ar[r] &1  }
\end{equation*}
where $\tilde{f}(m)=(f_1(m),0)=(0,f_2(m))$ and $\tilde{g}(a,b)=g_1(a)=g_2(b)$.

\medskip
Consider $(G_1,M_1,\phi_1)$ and $(G_2,M_2,\phi_2)$ two objects in $\mathbf{Tpl}_{Gr}$ and \
 $(\alpha, \beta) \in \Mor((G_1,M_1,\phi_1), (G_2,M_2,\phi_2))$. The induced homomorphism 
$$H^2(G_1,M_1,\phi_1) \rightarrow H^2(G_2,M_2,\phi_2)$$
is defined as follows: take a class in $H^2(G_1,M_1)$, that is, an extension,

$$1 \rightarrow M_1 \rightarrow \tilde{G_1} \rightarrow G_1 \rightarrow 1,$$
and consider the following diagram:

\begin{equation}
 \label{eq:H2Gr-2}
     \xymatrix{ 1  \ar[r] & M_1 \ar[rd]^{\tilde{i}_1} \ar[r]^{i_1}  &\tilde{G_1} \ar[r]^{\pi_1} &G_1 \ar[r] &1  \\
                 & & \tilde{H} \ar[r]^{\tilde{\pi_1}} \ar[u]^{\tilde{\alpha}} &G_2 \ar[u]^{\alpha} \ar[r] &1 }
\end{equation}
where 

$$\tilde{H}=\{(\tilde{g_1},g_2) \in \tilde{G_1} \times G_2 |  \pi_1(\tilde{\alpha}(\tilde{g_1},g_2))=\alpha(\tilde{\pi_1}(\tilde{g_1},g_2))\}$$
is the pull-back of the arrows $\tilde{G_1} \rightarrow G_1$ and $G_2 \rightarrow G_1$ and $\tilde{i}_1: M_1 \rightarrow \tilde{H}$ is defined by $\tilde{i}_1(m_1)=(i_1(m_1),1)$. It can be shown that in fact, $1 \rightarrow M_1 \rightarrow \tilde{H} \rightarrow G_2 \rightarrow 1$ is an exact sequence. 

\medskip
Similarly, construct the following diagram using the homomorphism $\beta: M_1 \rightarrow M_2$:

\begin{equation}
 \label{eq:H2Gr-3}
     \xymatrix{ 
         1 \ar[r] & M_1 \ar[d]^{\beta} \ar[r]^{\tilde{i_1}} & \tilde{H} \ar[d]^{\tilde{\beta}} \ar[r]^{\tilde{\pi_1}}  &G_2 \ar[r] &1 \\
                  1 \ar[r]& M_2 \ar[r]^{i_1} &\tilde{G_2}\ar[ur]_{\pi_2} & & }
\end{equation}
where 
$$\tilde{G}_2=\frac{M_2 \oplus \tilde{H}}{\{(\beta(m_1),-\tilde{i}_1(m_1)) | m_1 \in M_1\}}$$
is the push-out of the arrows $M_1\rightarrow M_2$ and $M_1\rightarrow \tilde{H}$ 
and $\pi_2((m_2,\tilde{h}))=\tilde{\pi}_1 (\tilde{h})$. It follows that 
 $$1 \rightarrow M_2 \rightarrow \tilde{G}_2 \rightarrow G_2 \rightarrow 1$$ 
 is a short exact sequence. This construction defines a homomorphism between $H^2(G_1,M_1,\phi_1)$ and $H^2(G_2,M_2,\phi_2)$.  In fact this defines a covariant functor $H^2(\cdot )$ between $\mathbf{Tpl}_{Gr}$ and $\mathbf{Ab}$.

\subsection{Defining $H^3$ in the category $\mathbf{Tpl}_{Gr}$}\label{sec:H3Gr} It is well-known that each class in the cohomology group $H^3(G;M)$ is in correspondence with a short exact sequence of the following form \cite{Br},

\begin{equation*}
\xymatrix{ 0 \ar[r] &M \ar[r] &H \ar[r] &\tilde{H} \ar[r] &G \ar[r] &1 }
\end{equation*}

\hspace{-6.0mm} where $0 \rightarrow M \rightarrow H_1 \rightarrow \tilde{H}_1 \rightarrow G \rightarrow 1$ and $0 \rightarrow M \rightarrow H_2 \rightarrow \tilde{H}_2 \rightarrow G \rightarrow 1$ are equivalent if there exist $f_1$ and $f_2$ that make the following diagram commute:

\begin{equation}\label{equivH3}
  \xymatrix{ 
          0 \ar[r] &M \ar[d]^{\id} \ar[r] &H_1 \ar[d]^{f_1} \ar[r] &\tilde{H}_1 \ar[d]^{f_2} \ar[r] &G \ar[d]^{\id} \ar[r] &1 \\
          0 \ar[r] &M \ar[r] &H_2 \ar[r] &\tilde{H}_2 \ar[r] &G \ar[r] &1}
\end{equation}

Constructing this short exact sequence is equivalent to saying that there is a crossed module $f: H \rightarrow \tilde{H}$, that is, $f$ is a homomorphism of groups together with an action of $\tilde{H}$ over $H$ denoted by $\eta: \tilde{H} \rightarrow \Aut(H)$ such that for $h_2, h_3 \in H$ and $h_1 \in \tilde{H}$

\begin{enumerate}
\item[(i)] $f(h_2.h_1)=h_1.f(h_2).h_1^{-1}=f(h_2)^{h_1}$
\item[(ii)] $f(h_2).h_3= h_2. h_3.h_2^{-1}=h_3^{h_2}$.
\end{enumerate}

Notice that in such a short exact sequence we are only able to control two out of four terms, namely, the nilpotency class of $G$ and the action length of $M$. The challenge is to keep the rest of the groups in our category of triples so that we can apply the correspondence of Lazard. 

Our first approach, however, will be the following. Fix a surjective homomorphism of groups $\alpha: G_1 \rightarrow G_2$ and a $G_2$-module $M$. Then, we consider all the crossed modules $\mu: G \rightarrow G_1$ that have M as the kernel of $\mu$ and $\alpha$ as the cokernel. We say that two crossed moduless $\mu: G \rightarrow G_1$ and $\mu': G' \rightarrow G_1$ are equivalent if there exists an isomorphism $f: G \rightarrow G'$ such that it is compatible with the actions of $G_1$ over $G$ and $G'$, $\mu' \circ f= \mu$ and $f_M =id_M.$ That is, the following commutes:

\begin{equation}
  \xymatrix{ 
          0 \ar[r] &M \ar[d]^{\id} \ar[r] &G \ar[d]^{f} \ar[r]^{\mu} &G_1 \ar[d]^{\id} \ar[r] &G_2 \ar[d]^{\id} \ar[r] &1 \\
          0 \ar[r] &M \ar[r] &G' \ar[r]^{\mu'} &G_1 \ar[r] &G_2 \ar[r] &1.}
\end{equation}

Denote by $\text{CMG}(G_2,G_1;M)$ the group of equivalence classes of all the crossed modules $\mu: G \rightarrow G_1$ with kernel $M$ and cokernel $\alpha: G_1 \rightarrow G_2$. Then, there is a one to one correspondence between $\text{CMG}(G_2, G_1,M)$ and the relative cohomology group $H^3(G_2,G_1;M)$. Observe that in this case, we control three out of four terms in the short exact sequence and thus, it will be easier to establish the necessary conditions to keep the reminding group in the category of modules so that we can apply the Lazard correspondence as mentioned before.

\begin{remark}
The relative degree three cohomology comes from the cochain complex $C^*(G_2,G_1;M)$ that fits in the following short exact sequence:

$$0 \rightarrow C^*(G_2;M) \rightarrow  C^*(G_1;M) \rightarrow C^*(G_2,G_1;M) \rightarrow 0.$$
\end{remark}

In order to give group structure to these crossed modules one needs to define an addition of such short exact sequences. As in the previous section, this is given by the Baer sum. Given two short exact sequences 

$$ 0 \rightarrow M \rightarrow G \rightarrow G_1 \rightarrow G_2 \rightarrow 1$$ and $$ 0 \rightarrow M \rightarrow G' \rightarrow G_1 \rightarrow G_2 \rightarrow 1$$ the Baer sum is described as follows: 

\begin{equation}{\label{BaerH3}}
\xymatrix{ 
& & & &0 \ar[d] & \\
& & & &M \ar@{.>}[dll]_{f'_2} \ar[d]^{f_2} & \\
 & &B' \ar[d]^{\tilde{p}_1} \ar[r]_{\tilde{p}_2} &A' \ar[d]^{\tilde{\pi}_1} \ar[r]_{\tilde{\pi}_2} &G' \ar[d]^{k_2} & \\
  &  &B \ar[d]^{p_1} \ar[r]^{p_2} &A \ar[d]^{\pi_1} \ar[r]^{\pi_2} &G_1 \ar[d]^{h_2} \ar[r]  &1\\
 0 \ar[r] &M \ar@{.>}[ruu]^{f'_1} \ar[r]^{f_1} &G \ar[r]^{k_1} &G_1 \ar[d] \ar[r]^{h_1} &G_2 \ar[d] \ar[r] &1  \\
& & &1 &1 & }
\end{equation}
where $A$ is the pull-back of $h_1$ and $h_2$; $A'$ is the pull-back of $\pi_2$ and $k_2$; $B$ is the pull-back of $\pi_1$ and $k_1$ and $B'$ is the pull-back of $\tilde{\pi}_1$ and $p_2$ containing $6$-tuples $(g, g_1, g_2, g_3, g_4, g') \in G \times G_1 \times G_1 \times G_1 \times G_1 \times G'$ such that $p_2(g, g_1,g_2)=(g_1,g_2)=(g_3,g_4)=\tilde{\pi}_1(g_3,g_4,g').$

Take 
$$\tilde{B} =\dfrac{B'}{\{(f_1(m), 0, 0, 0, 0, 0)-(0, 0, 0, 0, 0, f_2(m)) | m \in M \}}$$
and then the Baer sum is defined naturally by the next short exact sequence:

\begin{equation*}
\xymatrix{ 0 \ar[r] &M \ar[r]^{\tilde{f}} & \tilde{B} \ar[r]^{\tilde{k}_1} &G_1 \ar[r]^{\tilde{h}_1} &G_2 \ar[r] &0 
}
\end{equation*}
where $\tilde{f}(m)=(f_1(m), 0 ,0, 0, 0, 0)=(0, 0, 0, 0, 0, f_2(m))$, $\tilde{k}_1(g_1, g_2, g_3, g_4, g_5, g_6)=g_1$ and $\tilde{h_1}=h_1.$

\subsection{The subcategory $\mathbf{Tpl}_{Gr}^{c,d}$}

Let $(G,M,\phi)$ be a triple in $\mathbf{Tpl}_{Gr}$. Let $c$ denote the nilpotency class of $G$ and let $d$ denote the length of the action as follows:

Let $M_G=\frac{M}{[M,G]}$ where 

$$[M,G]=\text{Span}\{m-gm \; | \; g \in G, m \in M\}$$ 
and $[M,{}_{i}G]=[[M, {}_{i-1} G],G]$ for all $i \in \mathbb{N}$. Notice that 

$$M \supseteq [M,G] \supseteq [M, {}_2G] \supseteq \cdots \supseteq [M, {}_i G] \supseteq \cdots $$ 

Then, the smallest $d \in \mathbb{N}$ for which $[M, {}_d G]=1$ is called the length of the action of $G$ on $M$.

We say that a triple $(G,M,\phi )$ is contained in the subcategory $\mathbf{Tpl}_{Gr}^{c,d}$ of  $\mathbf{Tpl}_{Gr}$ if the nilpotency class of $G$ is at most $c$ and the length of the action of $G$ on $M$ is at most $d$.

\section{Lie algebras, modules and cohomology of low dimension}
\label{s:Lie}

Given a Lie ring $L$, we recall that an $L$-module $M$ is a $\ZZ_p$-module together with a homomorphism of Lie algebras $\phi : L\to \End (M)$. We define the category $\mathbf{Tpl}_{Lie}$ which takes as objects the triples $(L,M,\psi )$ where $L$ is a Lie ring, $M$ is a $\ZZ_p$-module and $\psi\in  \Hom (L,\End (M))$. For any two objects $(L_1,M_1,\psi_1 )$ and $(L_2,M_2,\psi_2 )$ in $\mathbf{Tpl}_{Lie}$, a morphism from $(L_1,M_1,\psi_1 )$ to $(L_2,M_2,\psi_2 )$ is given by a pair $(\alpha ,\beta )$ where $\alpha \in \Hom (L_2,L_1)$, $\beta \in  \Hom (M_1,M_2)$ and for all $a_2 \in L_2$ and $m_1 \in M,$ the following holds:

$$\beta ((\psi_1 \circ \alpha(a_2))(m_1) )= (\psi_2(a_2))(\beta(m_1)).$$

\subsection{Definition of $H^0$ in the category $\textbf{Tpl}_{Lie}$.} As in the previous section, given a triple $(L,M,\psi)$ in $\Obj(\mathbf{Tpl}_{Lie})$, one can define the cohomology group $H^0(L,M,\psi)$ as the module of invariant elements of $M$ under the $L$-action. That is,

$$H^0(L,M,\psi)=H^0(L,M)=M^L=\{m \in M \; | \; \psi(a)\cdot m=0, \; \forall a \in L\}$$

\medskip
Furthermore, if one has $(\alpha,\beta )$ a morphism between two triples $(L_1,M_1,\psi_1 )$ and $(L_2,M_2,\psi_2 )$, then 
$$\beta : H^0(L_1,M_1)\subseteq M_1\longrightarrow H^0(L_2,M_2)\subseteq M_2$$ 
is a homomorphism of abelian groups. In fact, 
$H^0(\cdot )$ defines a covariant functor between the category $\mathbf{Tpl}_{Lie}$ 
and the category of abelian groups $\mathbf{Ab}$.

\subsection{Defining $H^1$ in the category $\mathbf{Tpl}_{Lie}$} Given a triple $(L,M,\psi) \in \mathbf{Tpl}_{Lie}$, each equivalence class in $H^1(L,M,\psi)$ is in correspondence with the equivalent exact sequence of $L$-modules $0 \rightarrow M \rightarrow \tilde{M} \rightarrow \ZZ_p \rightarrow 0$. That is,

$$H^1(L,M,\psi)=\frac{\{\text{Equivalent extensions of}\quad L\text{-modules}\}}{\{\text{Equivalent split extensions of}\quad L\text{-modules}\}}.$$

The additive structure of $H^1(L,M, \psi)$ is defined as in the diagram \ref{eq:H1Gr-1}. In fact one has that $H^1(\cdot )$ is a covariant functor between $\mathbf{Tpl}_{Lie}$ and $\mathbf{Ab}$.

\subsection{Definition of $H^2$ in the category $\mathbf{Tpl}_{Lie}$} Each equivalence class of $H^2(L,M)$ classifies the equivalent extensions of $L$ by $M$ by considering the extensions of the form $0 \rightarrow M \rightarrow \tilde{L} \rightarrow L \rightarrow 0$ \cite{ChE}.That is,

$$H^2(L,M)= \frac{\{\text{Equivalent extensions of L by M}\}}{\{\text{Equivalent split extensions of L by M}\}}.$$

As for the extensions of groups, the sum of extensions of Lie algebras is given by the Baer sum as in \ref{eq:H2Gr-1}. In fact, $H^2(\cdot)$ is a covariant functor between $\mathbf{Tpl}_{Lie}$ and $\mathbf{Ab}$.

\subsection{Defining $H^3$ in the category $\textbf{Tpl}_{Lie}$}.

As in Section \ref{sec:H3Gr}, each class in the cohomology group $H^3(L;M)$ is in correspondence with a short exact sequence of the following form

$$0 \rightarrow M \rightarrow \mathbf{g} \rightarrow \tilde{\mathbf{g}} \rightarrow L \rightarrow 0$$

\hspace{-4.5mm}under the equivalence class defined as in $\eqref{equivH3}$.
Constructing this short exact sequence is equivalent to saying that there is a crossed module $f: \mathbf{g} \rightarrow \tilde{\mathbf{g}}$, that is, $f$ is a homomorphism of Lie algebras together with an action of $\tilde{\mathbf{g}}$ over $\mathbf{g}$ denoted by $\eta: \tilde{\mathbf{g}} \rightarrow \text{Der}(\mathbf{g})$ such that for $g_1,g_2 \in \mathbf{g}$ and $\tilde{g} \in \tilde{\mathbf{g}}$

\begin{enumerate}
\item[(i)] $f(\eta(\tilde{g}).g_1)=[\tilde{g}, f(g_1)]$
\item[(ii)] $\eta(f(g_1)).g_2= [g_1, g_2]$.
\end{enumerate}

Notice that in such a short exact sequence we are only able to control two out of four terms, namely, the nilpotency class of $L$ and the length action on $M$. The challenge is to keep the rest of the Lie algebras in our category of triples so that we can apply the correspondence of Lazard. 

Our first approach, however, will be as in the case for the groups. Fix a surjective homomorphism of Lie algebras $\alpha: L_1 \rightarrow L_2$ and a $L_2$-module $M$. Then, we consider all the crossed modules $\mu: L \rightarrow L_1$ that have M as the kernel of $\mu$ and $\alpha$ as the cokernel. We say that two crossed homomorphisms $\mu: L \rightarrow L_1$ and $\mu': L' \rightarrow L_1$ are equivalent if there exists an isomorphism $f: L \rightarrow L'$ such that it is compatible with the actions of $L_1$ over $L$ and $L'$, $\mu' \circ f= \mu$ and $f_L =id_L.$ That is, the following diagram commutes:

\begin{equation}
  \xymatrix{ 
          0 \ar[r] &M \ar[d]^{\id} \ar[r] &L \ar[d]^{f} \ar[r]^{\mu} &L_1 \ar[d]^{\id} \ar[r] &L_2 \ar[d]^{\id} \ar[r] &0 \\
          0 \ar[r] &M \ar[r] &L' \ar[r]^{\mu'} &L_1 \ar[r] &L_2 \ar[r] &0.}
\end{equation}

Denote by $\text{CML}(L_2,L_1;M)$ the group of the equivalence classes of all the crossed modules $\mu: L \rightarrow L_1$ with kernel $M$ and cokernel $\alpha: L_1 \rightarrow L_2$. Then, there is a one to one correspondence between $\text{CML}(L_2, L_1,M)$ and the relative cohomology group $H^3(L_2,L_1;M)$ as it is proven in the Appendix A of \cite{KL}. Observe that in this case, we control three out of four terms in the short exact sequence and thus, it will be easier to establish the necessary conditions to keep the reminding Lie algebra in the category of modules so that we can apply the Lazard correspondence as mentioned before.

\begin{remark}
As in the case of the cohomology of groups, the relative cohomology goup $H^3(L_2,L_1;M)$ comes from the cochain complex $C^*(L_2,L_1;M)$ that fits in the following exact sequence:

$$0 \rightarrow C^*(L_2,M) \rightarrow C^*(L_1,M) \rightarrow C^*(L_2,L_1;M) \rightarrow 0.$$
\end{remark}

The Baer sum of such two extensions is defined as in \eqref{BaerH3}.

\subsection{The subcategory $\mathbf{Tpl}_{Lie}^{c,d}$} Let $(L,M,\psi)$ be a triple in $\mathbf{Tpl}_{Lie}$. Let $c$ denote the nilpotency class of $L$ and let $d$ denote the length action defined as follows:

\medskip
Let $M_L=\frac{M}{[M,L]}$ where 

$$[M,L]=\text{Span}\{am \; | \; a \in L, m \in M\}$$
and $[M,{}_{i}L]=[[M, {}_{i-1} L],L]$. Notice that 

$$M \supseteq [M,L] \supseteq [M, {}_2L] \supseteq \cdots \supseteq [M, {}_i L] \supseteq \cdots$$

Then, the smallest $d \in \mathbb{N}$ for which $[M, {}_d L]=0$ is called the length action of $L$ on $M$.

\medskip
We say that a triple $(L,M,\psi )$ is contained in the subcategory $\mathbf{Tpl}_{Lie}^{c,d}$ of  $\mathbf{Tpl}_{Lie}$ if the nilpotency class of $L$ is at most $c$ and the length action of $L$ on $M$ is at most $d$.

\section{$\mathbf{Exp}$ and $\mathbf{Log}$ for triples} We define the exponential and logarithm maps for the objects $(G,M,\phi) \in \mathbf{Tpl}_{Gr}^{c,d}$ and $(L,M,\psi) \in \mathbf{Tpl}_{Lie}^{c,d}$ for $c,d <p$ as follows:

\medskip
$$\mathbf{Exp}: \; \mathbf{Tpl}_{Lie}^{c,d} \; \longrightarrow \; \mathbf{Tpl}_{Gr}^{c,d}$$
$$(L,M, \psi)\rightarrow(\mathbf{exp}(L),M, \text{exp}(\psi))$$ and 
$$\mathbf{Log}: \; \mathbf{Tpl}_{Gr}^{c,d} \; \longrightarrow \; \mathbf{Tpl}_{Lie}^{c,d}$$ 
$$(G,M, \phi)\rightarrow(\mathbf{log}(G),M, \text{log}(\phi))$$
 where $\displaystyle \mathop{\text{exp}(\psi)=\sum_{k=0}^{p-1}\frac{\psi^k}{k!}}$ and $\displaystyle \mathop{\text{log}(\phi)=\sum_{k=1}^{p-1}(-1)^{k+1}\frac{(\phi-\id)^k}{k}}.$

\bigskip
Similarly, one can also define $\mathbf{Exp}$ and $\mathbf{Log}$ for morphisms $(\alpha, \gamma) \in \Mor(\mathbf{Tpl}^{c,d}_{Gr})$ and $(\alpha, \beta) \in \Mor(\mathbf{Tpl}^{c,d}_{Lie})$ by $\mathbf{Exp}(\alpha,\gamma)=(\mathbf{exp}(\alpha),\gamma)$ and $\mathbf{Log}(\beta, \gamma)=(\mathbf{log}(\beta), \gamma).$ In the following lemmas we will show that these maps are in fact well-defined.

\begin{lemma} \label{4} Let $(G,M,\phi) \in \mathbf{Tpl}_{Gr}^{c,d}$ and $c,d<p$, then $(\mathbf{log}(G), M, \text{log}(\phi)) \in \mathbf{Tpl}_{Lie}^{c,d}$. Moreover, the following statements hold:
\begin{enumerate}
\item $(\phi- \id)(g)^d=0$ for all $g \in G$.
\item $[M, {}_i \textbf{log}(G)]\subseteq [M, {}_i G]$ for all $i \geq 1$.
\end{enumerate}
\end{lemma}

\begin{proof} 
We will start proving (1). Take $(G,M,\phi )\in \mathbf{Tpl}_{Gr}^{c,d}$. For all $g\in G$ and $m_1, m_2\in M$, we have
\begin{equation*}
\begin{split}
(\phi (g)-\id)(m_1+m_2)& =\phi (g)(m_1+m_2)-(m_1+m_2) \\ 
& = \phi (g)(m_1)-m_1 + \phi (g)(m_2)-m_2 \\
& =(\phi (g)-\id)(m_1) + (\phi (g) -\id)(m_2).
\end{split}
\end{equation*}
Therefore $(\phi (g)-id)\in \End (M)$. Furthermore, for any $G$-invariant submodule $U$ of $M$ one has by definition $(\phi (g) -id) (U)\subseteq [U,G]$. In particular $(\phi (g) -id)^d (M)\subseteq [M,_dG]=0$ and therefore $(\phi (g) -id)^d=0$.

We will continue by proving that $(\mathbf{log} (G),M,\log (\phi))\in \mathbf{Tpl}_{Lie}$. We would like to see that the map $\tilde{\phi}: G \rightarrow \Aut (M)$ given by $\tilde{\phi}(g)=\text{ log}(\phi(g))$ is also a homomorphism of Lie algebras so that we conclude that $M$ is also a $\mathbf{log}(G)$-module. We claim that for $g_1, g_2 \in G$, we have $\tilde{\phi}(g_1+g_2)= \tilde{\phi}(g_1) +\tilde{\phi}(g_2)$ and $\tilde{\phi}([g_1,g_2])=[\tilde{\phi}(g_1), \tilde{\phi}(g_2)]$. Indeed

\begin{equation*}
\begin{split} 
\tilde{\phi}(g_1+g_2)&= \log (\phi(g_1+g_2))=\log (\phi(h_1(g_1,g_2)))=
\log (h_1(\phi(g_1), \phi(g_2))) \\
& \quad =\log (\exp (\log (\phi(g_1))+\log (\phi(g_2)))) = \log (\phi(g_1)) + \log (\phi(g_2)) \\
& \quad = \tilde{\phi}(g_1) + \tilde{\phi}(g_2)
\end{split}
\end{equation*}

and

\begin{equation*}
\begin{split}
\tilde{\phi}([g_1,g_2])&= \text{log} (\phi([g_1,g_2]))= \text{log}(\phi(h_2(g_1,g_2)))= \text{log}(h_2(\phi(g_1),\phi(g_2))) \\
& \quad  =\text{log}(\exp ([\log (\phi(g_1)), \log (\phi(g_2))]))= [\tilde{\phi}(g_1),\tilde{\phi}(g_2)]
\end{split}
\end{equation*}
where $h_1(\cdot, \cdot)$ and $h_2(\cdot, \cdot)$ are the inverse Backer-Campbell- Hausdorff formulae. Therefore, $\tilde{\phi}$ is a homomorphism of Lie algebras and thus, $M$ is a $\mathbf{log}(G)$-module and $(\mathbf{log} (G),M,\log (\phi))\in \mathbf{Tpl}_{Lie}$.

For proving (2) notice that for a $G$-invariant submodule $U$ of $M$ we have that $U$ is also a $\mathbf {log} (G)$-module. Furthermore,  the action of $\mathbf{log} (G)$ in $U/[U,G]$ is trivial. Therefore $[U,\mathbf{log} (G)]\subseteq [U,G]$. An induction on $i$ now shows that $[M,_i\mathbf{log} (G)]\subseteq [M,_iG]$. 

In particular, we have  $(\mathbf{log}(G), M, \text{log}(\phi)) \in \mathbf{Tpl}_{Lie}^{c,d}.$
\end{proof}

\begin{lemma} \label{5} Let $(L,M, \psi) \in \mathbf{Tpl}_{Lie}^{c,d}$ and $c,d<p,$ then $(\mathbf{exp}(L),M,\text{exp}(\psi)) \in \mathbf{Tpl}_{Gr}^{c,d}.$ Moreover, the following conditions hold:
\begin{enumerate}
\item $\psi(a)^d=0$ for all $a \in L$.
\item $[M, {}_i \textbf{exp}(L)] \subseteq [M, {}_i L]$ for all $i \geq 1.$
\end{enumerate}
\end{lemma}

\begin{proof}
We will start proving (1). Take $(L,M,\psi )\in \mathbf{Tpl}_{Lie}^{c,d}$. We have that $\psi (a)\in \End (M)$. Furthermore, for any $L$-invariant submodule $U$ of $M$ one has $\psi (a) (U)\subseteq [U,L]$. In particular $(\psi (a) )^d (M)\subseteq [M,_dL]=0$ and therefore $(\psi (a))^d=0$.

We will continue by proving that $(\mathbf{exp} (L),M,\exp (\psi))\in \mathbf{Tpl}_{Lie}$. We would like to see that the map $\tilde{\psi}: L \rightarrow \End (M)$ given by $\tilde{\psi}(a)=\text{ exp}(\psi(a))$ is also a homomorphism of Lie algebras so that we conclude that $M$ is also a $\mathbf{exp}(L)$-module. We claim that for $a,b \in L$, $\tilde{\psi}(ab)= \tilde{\psi}(a)\tilde{\psi}(b)$ hold. Indeed

\begin{equation*}
 \begin{split}
\tilde{\psi}(ab) &=\text{exp}(\psi(ab))=\text{exp}(\psi(H(a,b)))=\text{exp}(H(\psi(a),\psi(b)))= \\
& \quad = \text{exp}(\text{log}(\text{exp}(\psi(a))\text{exp}(\psi(b))))=\text{exp}(\psi(a))\text{exp}(\psi(b))=\tilde{\psi}(a)\tilde{\psi}(b).
 \end{split}
\end{equation*}
where $H(\cdot,\cdot)$ denotes the Backer-Campbell-Hausdorff formula. Therefore, $\tilde{\psi}$ is a homomorphism of groups and thus, $M$ is a $\mathbf{exp}(L)$-module.

For proving (2) notice that for an $L$-invariant submodule $U$ of $M$ we have that $U$ is also an $\mathbf {exp} (L)$-module. Furthermore,  the action of $\mathbf{exp} (L)$ in $U/[U,L]$ is trivial. Therefore $[U,\mathbf{exp} (L)]\subseteq [U,L]$. An induction on $i$ now shows that $[M,_i\mathbf{exp} (L)]\subseteq [M,_iL]$. 

In particular, we have  $(\mathbf{exp}(L), M, \text{exp}(\psi)) \in \mathbf{Tpl}_{Gr}^{c,d}.$
\end{proof}

\medskip
\begin{theorem} \label{6} Let $c$ and $d$ be smaller than $p$. Then, 

$$\mathbf{Exp}: \mathbf{Tpl}_{Lie}^{c,d} \rightarrow  \mathbf{Tpl}_{Gr}^{c,d} \quad  \text{and} \quad \mathbf{Log}: \mathbf{Tpl}_{Gr}^{c,d} \rightarrow \mathbf{Tpl}_{Lie}^{c,d}$$ 
are well-defined functors. Moreover, these functors are isomorphisms of categories one inverse of the other. 
\end{theorem}

\begin{proof} By previous lemmas $\mathbf{Exp}$ and $\mathbf{Log}$ are well-defined functors. Moreover, 

\begin{equation*}
\begin{split}
(\mathbf{Log} \;\circ \;\mathbf{Exp})(L, M, \psi)&= \mathbf{Log}(\mathbf{exp}(L), M, \text{exp}(\psi)) \\
&=(\mathbf{log}(\mathbf{exp}(L)), M, \text{log}(\text{exp}(\psi))) \\
&=(L,M,\psi),
\end{split}
\end{equation*}

and

\begin{equation*}
\begin{split}
(\mathbf{Exp} \;\circ \;\mathbf{Log})(G,M,\phi)=&\mathbf{Exp}(\mathbf{log}(G), M, \text{log}(\phi)) \\
&=(\mathbf{exp}(\mathbf{log}(G)), M, \text{exp}(\text{log}(\phi))) \\
&=(G, M, \phi),
\end{split}
\end{equation*}
which completes the proof of the theorem.
\end{proof}

\begin{corollary} \label{7} Let $c,d<p$ and $(L,M,\psi) \in \mathbf{Tpl}_{Lie}^{c,d}.$ Then, $[M,{}_i \mathbf{exp}(L)]=[M, {}_i L]$ for all $i \geq 1$.
\end{corollary}

\begin{proof} This is just a direct consequence of the Lemma \ref{4} and Lemma \ref{5}.
\end{proof}

\section{Transporting cohomology}

In this section we would like to show the relation between the cohomology groups of a finite pro-$p$ group over a module $M$ and the cohomology groups of the corresponding Lie algebra over the same module $M$ for dimensions $0$, $1$, $2$ and $3$. In fact, we will show that after taking the $\mathbf{Exp}$ and $\mathbf{Log}$ functors defined in the previous section such cohomology functors are naturally equivalent.

\subsection{$H^0$ and $H_0$ for triples}

We will show that both $H^0$ and $H_0$ are the same for a Lie algebra and its corresponding finite pro-$p$ group.

\begin{theorem}Let $c,d <p$ and let
$$H^0_{Lie}: \mathbf{Tpl}_{Lie}^{c,d}\longrightarrow \mathbf{Ab}$$
and
$$H^0_{Gr}: \mathbf{Tpl}_{Gr}^{c,d}\longrightarrow \mathbf{Ab},$$
be the cohomology functors. Then 
$$H_{Lie}^0=H_{Gr}^0 \circ \mathbf{Exp}.$$
\end{theorem}

\begin{proof} 
Let $(L,M,\psi)$ be in $\mathbf{Tpl}_{Lie}^{c,d}$ and put $(G,M,\phi)=\mathbf{Exp}(L,M,\psi) \in \mathbf{Tpl}_{Gr}^{c,d}$. Recall that 

$$H_{Lie}^0(\mathbf{Log}(G,M,\phi))=\{m \in M \; | \; log(\phi)(a)\cdot m=0, \quad \text{for all} \quad a \in \mathbf{log}(G)\}.$$
Similarly,

$$H_{Gr}^0(\mathbf{Exp}(L,M,\psi))=\{m \in M \; | \; exp(\phi)(g)\cdot m=m, \quad \text{for all} \quad g \in \mathbf{exp}(L)\}.$$
We want to see that $H_{Lie}^0(L,M,\psi)=H_{Gr}^0(G,M,\phi)$.

\medskip
\noindent $\mathbf{Subclaim \; 1:}$ $H_{Gr}^0(G,M,\phi) \subset H_{Lie}^0(\mathbf{Log}(G,M,\phi)).$ 

\medskip

\begin{proof}[Subproof] Let $m \in H_{Gr}^0(G,M,\phi)$, that is, $\phi(g)\cdot m=m$ for all $g \in G$. Then, 

$$\displaystyle \mathop{\log(\phi)(g)\cdot m=\sum_{k=1}^{p-1} \frac{(\phi-id)(g)^k}{k}\cdot m}=0$$
for all $g\in G$. Hence, $m \in H_{Lie}^0(\mathbf{Log}(G,M,\phi))$.
\end{proof}

\medskip
\noindent $\mathbf{Subclaim \; 2:}$ $H_{Lie}^0(L,M,\psi)\subset H_{Gr}^0(\mathbf{Exp}(L,M,\psi)).$

\begin{proof}[Subproof] Let $m \in H_{Lie}^0(L,M,\psi)$, that is, $\psi(a)\cdot m=0$ for all $a \in L$. Then, 

$$\displaystyle \mathop{\exp(\psi)(a)\cdot m=\sum_{k=0}^{p-1}\frac{\psi(a)^k}{k!}\cdot m}=m$$
for all $a \in L$. Hence, $m \in H_{Gr}^0(\mathbf{Exp}(L,M,\psi)).$
\end{proof}

Since $\mathbf{Exp}$ and $\mathbf{Log}$ are isomorphisms of categories, one inverse of the other, and by the first subclaim, we have

$$H_{Gr}^0(\mathbf{Exp}(L,M,\psi)) \subset H_{Lie}^0(\mathbf{Log}(\mathbf{Exp}(L,M,\psi)))=H_{Lie}^0(L,M,\psi).$$

Now the equality of the functors is clear by Subclaim 2.
\end{proof}

\begin{theorem} Let $c,d<p$ and let
$$H_0^{Lie}: \mathbf{Tpl}_{Lie}^{c,d}\longrightarrow \mathbf{Ab}$$
and
$$H_0^{Gr}: \mathbf{Tpl}_{Gr}^{c,d}\longrightarrow \mathbf{Ab},$$
be the homology functors. Then 

$$H^{Lie}_0=H^{Gr}_0 \circ \mathbf{Exp}.$$
\end{theorem}

\begin{proof} 
Let $(L, M, \psi)$ be in $\mathbf{Tpl}_{Lie}^{c,d}$ and put $(G,M,\phi)=\mathbf{Exp}(L,M,\psi)$. Since $ \mathbf{Exp}$ is an isomorphism of categories between $\mathbf{Tpl}_{Lie}^{c,d}$ and $\mathbf{Tpl}_{Gr}^{c,d}$, it is enough to prove that $H^{Lie}_0(L,M,\psi)=H^{Gr}_0(\mathbf{Exp}(L,M,\psi))$.

By definition $H_0^{Gr}= M_G$ and $H_0^{Lie}=M_L$. By Corollary \ref{7} we know that $[M,G]=[M,L]$. Then, as $M_G=\frac{M}{[M,G]}$, $M_L=\frac{M}{[M,L]}$, the equality $H^{Lie}_0(L,M,\psi)=H^{Gr}_0(\mathbf{Exp}(L,M,\psi))$ holds from the fact that $[M,G]=[M,L]$ by Corollary \ref{7}. 
\end{proof}

\subsection{$\mathbf{Exp}$ and $\mathbf{Log}$ for exact sequences of modules and $H^1$} 

We will show that the classes in $H^1$ remain unchanged after applying the $\mathbf{Exp}$ and $\mathbf{Log}$ functors.

\begin{proposition} \label{10} Let $c,d<p$ and $(L,M_1,\phi_1), (L,M_2,\phi_2),(L,M_3, \phi_3) \in \mathbf{Tpl}_{Lie}^{c,d}.$ Then,

\begin{displaymath}
         \xymatrix{ 0 \ar[r] & M_1 \ar[r]^{\alpha} &M_2 \ar[r]^{\beta} &M_3 \ar[r] &0 }
\end{displaymath}
is an exact sequence of $L$-modules if and only if
\begin{displaymath}
         \xymatrix{ 0 \ar[r] & M_1 \ar[r]^{\alpha} &M_2 \ar[r]^{\beta} &M_3 \ar[r] &0 }
\end{displaymath}
is an exact sequence of $\mathbf{exp}(L)$-modules.
\end{proposition}

\begin{proof} Notice that the action of $L$ over the $M_i$ modules commutes with $\alpha$ and $\beta$ and thus, by theorem \ref{6} the same will happen with the action of $\mathbf{exp}(L)$ over such modules. This fact proves the proposition.
\end{proof}

\begin{theorem} {\label {H1}} Let $c<p$, $d<p-1$ and let
$$H^1_{Lie}: \mathbf{Tpl}_{Lie}^{c,d}\longrightarrow \mathbf{Ab}$$
and
$$H^1_{Gr}: \mathbf{Tpl}_{Gr}^{c,d}\longrightarrow \mathbf{Ab},$$
be the cohomology functors. Then $H_{Lie}^1$ and $H_{Gr}^1 \circ \mathbf{Exp}$ are naturally equivalent and the natural transformation that provides the equivalence of functors is the one given by Proposition \ref{10}.
\end{theorem}

\begin{proof}
Take $(L,M,\psi)\in   \mathbf{Tpl}_{Lie}^{c,d}$, and
 let $0 \rightarrow M \rightarrow \tilde{M} \rightarrow \mathbb{Z}_p \rightarrow 0$ be
an element in $H_{Lie}^1(L,M,\psi)$. This means that $\frac{\tilde{M}}{M} \cong \mathbf{Z}_p$ and thus, $[\tilde{M}, L] \leq M$. Then, $[\tilde{M}, _{d+1} L]=0.$ That is, $(L,\tilde{M},\tilde{\psi}) \in \mathbf{Tpl}_{Lie}^{c,d+1}$ and Lazard correspondence holds as $d < p-1$. We have a similar argument if we start with a triple in $\mathbf{Tpl}_{Gr}^{c,d}$.

Notice that the Baer sum of an extension of $G$-modules is the Baer sum of the corresponding $L$-module extensions which follows from the diagram \eqref{eq:H1Gr-1}. This shows that  $H_{Lie}^1(L,M,\psi)\cong (H_{Gr}^1 \circ \mathbf{Exp})(L,M,\psi)$. 

Finally, given a morphism $(\alpha ,\beta) \in \Mor(\textbf{Tpl}^{c,d}_{Lie})$, the morphism $H_{Lie}^1(\alpha ,\beta)$ is transformed into the morphism  $H_{Gr}^1 \circ \mathbf{Exp} (\alpha ,\beta)$ by the diagram \eqref{eq:H1Gr-2}.
\end{proof}

\subsection{Exp and Log for group extensions and $H^2$.} 
Let $(G,M,\phi) \in \mathbf{Tpl}_{Gr}^{c,d}$ and consider
$$1 \rightarrow M \rightarrow \tilde{G} \rightarrow G \rightarrow 1.$$
Then the nilpotency class of $\tilde{G}$ is at most $c+d$. Similarly if  $(L,M,\psi) \in \mathbf{Tpl}_{Lie}^{c,d}$ and 
$$0 \rightarrow M \rightarrow \tilde{L} \rightarrow L \rightarrow 0$$
is an extension of Lie algebras, the nilpotency class of $\tilde{L}$ is at most $c+d$.
The following result is now clear from Lazard correspondence.

\begin{proposition}{\label{pro:H2}} Let $c+d <p$ and $(L,M,\psi) \in \mathbf{Tpl}_{Lie}^{c,d}$. Then,

$$0 \rightarrow M \rightarrow \tilde{L} \rightarrow L \rightarrow 0$$
is an extension of the Lie algebra $L$ if and only if 
$$1 \rightarrow M \rightarrow \mathbf{exp}(\tilde{L}) \rightarrow \mathbf{exp}(L) \rightarrow 1$$ is an extension of the group $\mathbf{exp}(L)$. 
\end{proposition}

Now, we are ready to show the next result.

\begin{theorem} \label{H2} Let $c+d<p$ and let
$$H^2_{Lie}: \mathbf{Tpl}_{Lie}^{c,d}\longrightarrow \mathbf{Ab}$$
and
$$H^2_{Gr}: \mathbf{Tpl}_{Gr}^{c,d}\longrightarrow \mathbf{Ab},$$
be the cohomology functors. Then $H_{Lie}^2$ and $H_{Gr}^2 \circ \mathbf{Exp}$ are naturally equivalent and the natural transformation that provides the equivalence of functors is given in Proposition \ref{pro:H2}.
\end{theorem}

\begin{proof}
Notice that by Proposition \ref{pro:H2} extensions and split extensions are preserved by the  $\textbf{exp}$-$\textbf{log}$ functors. By diagram \eqref{eq:H2Gr-1}, the Baer sum is also preserved. Finally one can see that by diagrams  \eqref{eq:H2Gr-2} and \eqref{eq:H2Gr-3} the morphisms are transformed.
\end{proof}

\subsection{Exp and Log for group extensions and $H^3$.} 

Let $(G,M,\phi)$ and $(G_2, M, \phi_2)$ be in $\mathbf{Tpl}_{Gr}^{c,d}$ with nilpotency classes $\tilde{c} < c$ and $c_2 < c$ respectively and length action $d$. Consider $0 \rightarrow M \rightarrow G_1 \rightarrow G_2 \rightarrow G \rightarrow 1.$ Then, the nilpotency class of $G_1$ is at most $c_2 +d$ and thus, $(G_1,M,\phi_1) \in \mathbf{Tpl}_{Gr}^{c_2+d,d}$.

Similarly, if $(L,M,\psi),(L_2,M,\psi_2) \in \mathbf{Tpl}^{c,d}_{Lie}$ with nilpotency classes $\tilde{c} < c$ and $c_2 < c$, respectively, the length action is $d$ and if we consider $1 \rightarrow M \rightarrow L_1 \rightarrow L_2 \rightarrow L \rightarrow 1$, then $(L_1, M \psi_1) \in \mathbf{Tpl}^{c_2+d,d}_{Lie}.$ 

\begin{proposition}{\label{cross3}} Let $c+d <p$, $(G,M,\phi), (G_2,M,\phi_2)\in \mathbf{Tpl}_{Gr}^{c,d}$. Then,

$$0 \rightarrow M \rightarrow G_1 \rightarrow G_2 \rightarrow G \rightarrow 1$$
is a crossed module for $\nu: G_2 \rightarrow G$ and a $G$-module $M$ if and only if

$$0 \rightarrow M \rightarrow \textbf{log}(G_1) \rightarrow \textbf{log}(G_2) \rightarrow \textbf{log}(G) \rightarrow 0$$
is a crossed module for $\textbf{log}(\nu): \textbf{log}(G_2) \rightarrow \textbf{log}(G)$ and the $\textbf{log}(G)$-module $M$. In particular 
$$H_{RGr}^3(G,G_2;M)\cong H_{RLie}^3(\textbf{log} (G),\textbf{log}(G_2);M).$$
\end{proposition} 

\begin{proof} Let 
\begin{displaymath}
         \xymatrix{ 0 \ar[r] &M \ar[r]^{f} &G_1 \ar[r]^{k} &G_2 \ar[r]^{h} &G \ar[r] &1
}
\end{displaymath}
be a crossed module with $(G,M,\phi)$ and $(G_2,M,\psi)$  in $\mathbf{Tpl}_{Gr}^{c,d}$.

By definition, $G_2$ acts on $G_1$ and we can consider $M$ to be contained in $G_1$. Then, the length action of $G_1$ over $M$ is at most $d$ as 

$$[M, {}_i G_1] \leq [M,{}_i G_2] \; \text{for all} \; i\geq 1.$$

It is enough to show that the nilpotency class of $G_1$ is at most $p$. Indeed, $\frac{G_1}{\ker(k)}\simeq \im(k) \subset G_2$ and thus $\gamma_{c+1}(G_1) \subset \ker(k)= M$. Hence $\gamma_{c+d+1}(G_1) \leq [M, {}_d G_2]=1$. By hypothesis, $c+d <p$ and thus, the Lazard correspondence holds, that is, %

$$0 \rightarrow M \rightarrow \textbf{log}(G_1) \rightarrow \textbf{log}(G_2) \rightarrow \textbf{log}(G) \rightarrow 0$$
is a crossed module for $\mathbf{log}(\nu): \mathbf{log}(G_2) \rightarrow \mathbf{log}(G)$.

Similarly, one can give the same arguments starting with an extension of Lie algebras. The last statement of the proposition follows from diagram $\eqref{BaerH3}$.
\end{proof}

\section{Applications}

Let $G$ be a finite $p$-group and $L$ be a Lie algebra. We have the following results based on Theorem  \ref{H2}.

\begin{theorem}Let $\textbf{Gr}_{p-2}$ be the category of a finite $p$-group of nilpotency class smaller than $p-1$ and $\textbf{Lie}_{p-2}$ the category of finite and nilpotent $\ZZ_p$-Lie algebras of nilpotency class smaller than $p-1$. Denote by 
\begin{eqnarray*}
  \M_{Gr} & : & \textbf{Gr}_{p-2}\longrightarrow \textbf{Ab} \ \ \ \text{and} \\
  \M_{Lie} & :  & \textbf{Lie}_{p-2}\longrightarrow \textbf{Ab} 
\end{eqnarray*}
the group and Lie algebra Schur Multiplier functors, respectively. Then 
$\M_{Gr}\circ \textbf{exp}$ and $\M_{Lie}$ are naturally equivalent. In particular, for $L\in \textbf{Lie}_{p-2}$ one has $\M_{Gr}(\textbf{exp}(L))\cong \M_{Lie} (L)$.
\end{theorem}

\begin{proof}
Let $G$ be a finite $p$-group of nilpotency class smaller than $p-1$ and write $L=\textbf{log} (G)$. Observe that on the one hand $H_{Gr}^2(G, C_{p^\infty})= \varinjlim H_{Gr}^2(G, C_{p^i})$ and $H_{Lie}^2(L, C_{p^\infty})= \varinjlim H_{Lie}^2(L, C_{p^i})$ as  $C_{p^{\infty}}=\varinjlim (C_p\hookrightarrow C_{p^2} \hookrightarrow C_{p^3}\ldots)$. On the other hand, by theorem \ref{H2} there exists an isomorphism $\tau_i:H_{Gr}^2(G,C_{p^i})\longrightarrow H_{Lie}^2(L,C_{p^i})$ which is given by the natural transformation between the functors $H_{Gr}^2$ and $H_{Lie}^2$ . The isomorphism 
$\tau_\infty=\varinjlim \tau_i : H^2(G,C_{p^\infty})\longrightarrow H^2(L,C_{p^\infty})$ defines a natural equivalence between $\M_{Gr}$ and $\M_{Lie}$.
\end{proof}

The exact sequences of cohomology groups give another way of computing the cohomology of some groups. For example, the inflation-restriction-transgression exact sequence comes from studying spectral sequences. We want to show that this exact sequence commutes with the functors $\textbf{exp}$ and $\textbf{log}$.

Consider $\FF_p$ with trivial group action and denote the cohomology groups by $H_{Gr}^*(G)$ and $H^*_{Lie}(L)$ for short. Let $G$ be a finite group with nilpotency class $c$ and $N \unlhd G$. We use the next characterization of $H_{Gr}^1(\cdot)$ in \cite{Ev} to simplify the computations. Define $H_{Gr}^1(G)= \Hom(G, \FF_p)$ and $H_{Gr}^1(N)^{\frac{G}{N}}=\Hom(\frac{N}{N^p[G,N]},\FF_p)$. Analogously for a Lie algebra $L$ and a Lie ideal $I$ one has that $H_{Lie}^1(L)=\Hom(L, \FF_p)$ and $H_{Lie}^1(I)^{\frac{L}{I}}=\Hom(\frac{I}{pI+[L,I]},\FF_p)$.

\begin{lemma}{\label{tr}} Let $G$ be a finite $p$-group of nilpotency class smaller than $p$ and $N$ a normal subgroup of $G$. Let $L=\textbf{log} (G)$ and $I=\textbf{log} (N)$ and denote by $\mathbf{tr}_{Gr}: H_{Gr}^1(N)^{\frac{G}{N}} \rightarrow H_{Gr}^2(\frac{G}{N})$ and by $\mathbf{tr}_{Lie}: H_{Lie}^1(I)^{\frac{L}{I}} \rightarrow H_{Lie}^2(\frac{L}{I})$ the transgression maps. Then the following diagram is commutative:
\begin{equation}
  \xymatrix{ 
          H^1_{Gr}(N)^{\frac{G}{N}} \ar[r]^{tr_{Gr}} \ar[d]^{\tau^1_N}&H^2_{Gr}(\frac{G}{N}) \ar[d]^{\tau^2_{\frac{G}{N}}}\\
                    H^1_{Lie}(I)^{\frac{L}{I}} \ar[r]^{tr_{Lie}} &H^2_{Lie}(\frac{L}{I}), 
                    }
    \end{equation}
where $\tau^1_N$ and $\tau^2_{\frac{G}{N}}$ are the isomorphism given by the natural equivalence of functors.
\end{lemma}

\begin{proof}The transgression map $tr_{Gr}$ is defined as follows: let $f \in H^1(N)^{\frac{G}{N}} =\Hom(\frac{N}{N^p[G,N]},\FF_p)$ and consider $M \leq G$ such that $\ker(f)=\frac{M}{N^p[G,N]}$. Notice that $\frac{N}{M} \cong \FF_p$ Then, the image of $f$ is the next exact sequence in $H^2(\frac{G}{N})$:

\begin{equation*}
\xymatrix{1 \ar[r] &\frac{M}{N} \ar[r] &\frac{G}{M} \ar[r] &\frac{G}{N} \ar[r] &1.
}
\end{equation*}

If $c < p-1$, it is clear that the $\mathbf{log}$ of this exact sequence 
$$0 \rightarrow \frac{\textbf{log}(N)}{\textbf{log}(M)} \rightarrow \frac{\textbf{log}(G)}{\textbf{log}(M)} \rightarrow \frac{\textbf{log}(G)}{\textbf{log}(N)} \rightarrow 0$$
is the image of $f \in H^1(\textbf{log}(N))^{\textbf{log}(\frac{G}{N})}$ by the transgression map $\mathbf{tr}_{Lie}$.
\end{proof}

A consequence of this lemma is the next result.

\begin{proposition} Suppose that $c < p-1$. Then the following diagram commutes:

\begin{equation*}
\xymatrix{0 \ar[r] &H_{Gr}^1(\frac{G}{N}) \ar[d]^{\tau^1_{\frac{G}{N}}}\ar[r]  &H_{Gr}^1(G) \ar[d]^{\tau^1_G} \ar[r] &H_{Gr}^1(N)^{\frac{G}{N}} \ar[d]^{\tau^1_N} \ar[r]^{tr_{Gr}} &H_{Gr}^2(\frac{G}{N}) \ar[d]^{\tau^2_{\frac{G}{N}}} \ar[r] &H_{Gr}^2(G)\ar[d]^{\tau^2_G}\\
0 \ar[r] &H_{Lie}^1(\frac{L}{I}) \ar[r]  &H_{Lie}^1(L) \ar[r] &H_{Lie}^1(I)^{\frac{L}{I}} \ar[r]^{tr_{Lie}} &H_{Lie}^2(\frac{L}{I}) \ar[r] &H_{Lie}^2(L) }
\end{equation*}
where $L= \textbf{log}(G)$ and $I= \textbf{log}(N)$ and $\tau^i$ are the isomorphisms given by the natural equivalence of functors.
\end{proposition}

\begin{proof} This is is a direct consequence of Theorem $\ref{H1}$, Theorem $\ref{H2}$ and Lemma $\ref{tr}$.
\end{proof}

\end{document}